\documentclass[12pt]{amsart} 

\usepackage{amsmath} \usepackage{amssymb} \usepackage{amsthm}
\usepackage{amscd} 
\usepackage{enumerate}
\usepackage{mathrsfs}
\usepackage[dvipdfm]{hyperref}
\usepackage{mathtools}
\usepackage{graphicx}
\hypersetup{colorlinks=true}

\def\Z{{\mathbb{Z}}}
\def\N{{\mathbb{N}}}
\def\rN{{\mathrm{N}}}

\def\Q{{\mathbb{Q}}} 
\def\R{{\mathbb{R}}} 

\def\B{{\bf{B}}}
\def\fa{{\mathfrak{a}}} 
\def\fb{{\mathfrak{b}}}

\def\m{{\mathfrak{m}}}

\def\vt{{\vec{t}}}
 
\def\J{{\mathcal{J}}} 
\def\sO{{\mathcal{O}}}
\def\sL{{\mathscr{L}}}
\def\D{{\Delta}}
\def\e{{\mathrm{e}}}
\newcommand{\cen}{c}

\def\Hom{{\mathrm{Hom}}} 
\def\SHom{{\mathscr{H} \kern -.5pt om}}

\def\Spec{{\mathrm{Spec\; }}}

\def\Div{{\mathrm{div}}}
\def\ord{{\mathrm{ord}}} 
\def\mult{{\mathrm{mult}}}
\def\Tr{{\mathrm{Tr}}}
\def\Ram{{\mathrm{Ram}}}

\def\Supp{{\mathrm{Supp}}}
\def\NNA{{\Gamma}}
\def\NNef{{\mathrm{NNef}}}
\def\Frac{{\mathrm{Frac}}}
\def\num{{\mathrm{num}}}

\theoremstyle{plain}
\newtheorem{thm}{Theorem}[section] 
\newtheorem{cor}[thm]{Corollary}
\newtheorem{prop}[thm]{Proposition}
\newtheorem{conj}[thm]{Conjecture}
\newtheorem*{mainthm}{Main Theorem}

\newtheorem{lem}[thm]{Lemma}
\theoremstyle{definition} 
\newtheorem{defn}[thm]{Definition}
\newtheorem{propdef}[thm]{Proposition-Definition} 
\newtheorem{eg}[thm]{Example} 
\theoremstyle{remark}
\newtheorem{rem}[thm]{Remark}

\newtheorem{setup}[thm]{Set up}

\newtheorem*{cl}{Claim}

\newtheorem*{acknowledgement}{Acknowledgments}

\title{Stability of test ideals of divisors with small multiplicity}
\author{Kenta Sato}
\address{Graduate School of Mathematical Science, the University of Tokyo, 3-8-1 Komaba, Meguro-ku, Tokyo 153-8914, Japan}
\email{ktsato@ms.u-tokyo.ac.jp}

\begin{document}
\tolerance = 9999

\maketitle
\markboth{KENTA SATO}{STABILITY OF TEST IDEALS OF DIVISORS WITH SMALL MULTIPLICITY}

\begin{abstract}
Let $(X, \Delta)$ be a log pair in characteristic $p>0$ and $P$ be a (not necessarily closed) point of $X$.
We show that there exists a constant $\delta>0$ such that 
the test ideal $\tau(X, \Delta)$, a characteristic $p$ analogue of a multiplier ideal, does not change at $P$ under the perturbation of $\Delta$ by any $\R$-divisor with multiplicity less than $\delta$.
As an application, we prove that if $D$ is an $\mathbb{R}$-Cartier $\mathbb{R}$-divisor on a strongly $F$-regular projective variety,
then the non-nef locus of $D$ coincides with the restricted base locus of $D$.
This is a generalization of a result of Musta\c{t}\v{a} to the singular case and can be viewed as a characteristic $p$ analogue of a result of Cacciola--Di Biagio.
\end{abstract}

\setcounter{tocdepth}{1}


\section{Introduction}
A scheme $X$ of positive characteristic is called \emph{$F$-finite} if the Frobenius morphism $F:X \to X$ is a finite morphism.
Let $X$ be an $F$-finite Noetherian normal scheme and $\D$ be an effective $\R$-Weil divisor on $X$.
The \emph{test ideal} $\tau(X, \D)$, which is defined in terms of the Frobenius morphism as in \cite{Tak}, satisfies several important properties similar to those of a multiplier ideal.
One such property is the stability of test ideals under small perturbations: 
for an effective $\R$-Cartier $\R$-divisor $D$ on $X$ and a (not necessarily closed) point $P \in \Supp(D)$, there exists a constant $c=c(X, \D, P ;D)>0$ such that
\begin{eqnarray*}
\tau(X, \D+ \epsilon D)_P  =  \tau(X, \D)_P  &  & (\textup{if }0 \le \epsilon < c) \\
\tau(X, \D+ \epsilon D)_P  \subsetneq  \tau(X, \D)_P  & & (\textup{if }c \le \epsilon).
\end{eqnarray*}
Such a constant $c$ is called the \emph{smallest $F$-jumping number} of $D$ with respect to $(X, \D, P)$ and encodes much information on the singularities of $D$.

As another invariant which measures the singularities of $D$, we consider the \emph{multiplicity} $\mult_P(D)$ of $D$ at $P$ (see Definition \ref{multiplicity}).
In characteristic $0$, it is known that the multiplier ideal $\J(V, B)$ is uniformly stable under the perturbation of $B$ by any $\R$-divisor with small multiplicity.
For example, if $x \in V$ is a point of a smooth complex variety $V$ and $B$ is an effective $\R$-divisor with $\mult_x(B)<1$,
then the multiplier ideal $\J(V,B)$ is trivial at $x$ (\cite[Proposition 9.5.13]{Laz}).
A more general statement was proved by Li (\cite[Second Proof of Proposition 2.3]{Li}).
Motivated by those results, we prove an analogous statement in positive characteristic using test ideals instead of multiplier ideals.

\begin{mainthm}[\textup{Corollary \ref{main'''}}]
	Let $X$ be an $F$-finite Noetherian normal scheme, $\D$ be an $\R$-Weil divisor on $X$, and $P$ be a point of $X$.
	Then there exists $\delta>0$ such that
        $\tau(X, \D+ E)_P=\tau(X, \D)_P$ 
        for every effective $\R$-Cartier $\R$-divisor $E$ on $X$ with $\mult_P(E)< \delta$.
\end{mainthm}

The proof is based on a description of test ideals in terms of the trace map for the Frobenius morphism and a perturbation trick (see Section 3 for details).

We use the main theorem to give a partial answer to a conjecture on the non-nef locus of an $\R$-Cartier $\R$-divisor on a projective variety.
Let $X$ be a normal projective variety over an arbitrary field and $D$ be an $\R$-Cartier $\R$-divisor on $X$.
The \emph{non-nef locus} $\NNef(D)$ is defined in terms of the numerical vanishing orders of $D$ along divisorial valuations (cf. Definition \ref{v_{num}} and \ref{NNef} below).
As the name suggests, $\NNef(D) = \emptyset$ if and only if $D$ is nef.
On the other hand, the \emph{restricted base locus} $\B_-(D)$ is defined as a lower approximation of the stable base locus $\B(D) \coloneqq \bigcap_{E \sim_{\Q} D, E \ge 0} \Supp(E)$ obtained by perturbations of $D$ by small ample divisors.
It has the same property as the non-nef locus, that is, $\B_-(D) = \emptyset$ if and only if $D$ is nef.
Although the definitions of $\B_-(D)$ and $\NNef(D)$ seem to be far different, the following was conjectured.

\begin{conj}[\textup{\cite[Conjecture 2.7]{BBP}}]\label{conj}
	Let $X$ be a normal projective variety over an algebraically closed field and $D$ be an $\R$-Cartier $\R$-divisor on $X$.
        Then we have
	\[\NNef(D)=\B_-(D).\]
\end{conj}

In characteristic zero, Conjecture \ref{conj} was proved by Nakayama \cite[V, Lemma 1.9 (1)]{Nak} if $X$ is smooth and recently by Cacciola--Di Biagio \cite[Corollary 4.9]{CdB} if $X$ has only klt singularities.
In positive characteristic, Musta\c{t}\v{a} \cite{Mus} proved the smooth case, 
but his proof heavily depends on the subadditivity property of test ideals, which can fail on singular varieties.
In this paper, assuming that $X$ has strongly $F$-regular singularities instead of klt singularities, 
we prove a characteristic $p$ analogue of the result of Cacciola--Di Biagio.

\begin{thm}[\textup{Corollary \ref{Nakayama}}]\label{Nak}
	Let $X$ be a normal projective variety over an $F$-finite field.
        If the non-strongly $F$-regular locus of $X$ has dimension at most zero, then Conjecture \ref{conj} holds.
        That is,
	\[\B_-(D)= \NNef(D)\]
        for every $\R$-Cartier $\R$-divisor $D$ on $X$.
\end{thm}

Theorem \ref{Nak} is obtained from the main theorem as follows.
For simplicity, we assume that $X$ is $\Q$-Gorenstein strongly $F$-regular and $D$ is a big $\Q$-Cartier $\Q$-divisor.
Suppose that a point $x \in X$ is not contained in the non-nef locus $\NNef(D)$.
It then follows from the main theorem that the asymptotic test ideal $\tau(X, m \cdot ||D||)$ is trivial at $x$ for all $m>0$.
On the other hand, there exists an ample divisor $A$ on $X$ such that $\tau(X, m \cdot ||D||) \otimes \sO_{X}(mD+A)$ is globally generated for all $m>0$ (see Proposition \ref{grob} for the precise statement).
Therefore, $\sO_{X}(mD+A)$ is globally generated at $x$, which implies $x \not\in \B_{-}(D)$.

\begin{small}
\begin{acknowledgement}
The author wishes to express his gratitude to his supervisor Professor Shunsuke Takagi for his encouragement, valuable advice and suggestions. 
He is grateful to Doctor Sho Ejiri for his encouragement.
He is also grateful to an anonymous referee for many useful suggestions and for pointing out many typos.
This work was supported by the Program for Leading Graduate Schools, MEXT, Japan.
\end{acknowledgement}
\end{small}

\section{Preliminaries}

In this section, we recall definitions and basic facts which we will need later.

\subsection{Trace maps}\label{Trace map}
In this subsection, we recall the trace maps for alterations. 
The reader is referred to \cite{BST} for details.

Let $X,Y$ be finite dimensional Noetherian normal connected schemes and $f:Y \to X$ be an alteration, that is, a generically finite proper dominant morphism.
Assume that $X$ has a canonical module $\omega_X$.
In this situation, $Y$ also has a canonical module $\omega_Y$.

\begin{rem}\label{f^!}
Although canonical modules on $Y$ are not unique, there is a canonical choice of $\omega_Y$ determined by $f$ and $\omega_X$ (\cite[p.7]{BST}).
\end{rem}

Let $\omega_Y$ be as in Remark \ref{f^!}.
There exists a \emph{trace map} $\Tr_f: f_*\omega_Y \to \omega_X$.
Trace maps are compatible with composition of alterations (\cite[Lemma 2.3]{BST}) and when $f$ is birational or finite, $\Tr_f$ is characterized by the example below.

\begin{eg}[\textup{\cite[p.4]{BST}}]\label{tr}
Let $f: Y\to X$ be an alteration of finite dimensional Noetherian normal connected schemes, 
$\omega_X$ be a canonical module on $X$, $K_X$ be a canonical divisor on $X$ such that $\omega_X \simeq \sO_{X}(K_X)$, and $\omega_Y$ be as in Remark \ref{f^!}.
\begin{enumerate}
        \item Assume that $f$ is birational. 
        There is a unique canonical divisor $K_Y$ on $Y$ which coincides with $K_X$ on the locus where $f$ is isomorphism.
	Then there is a commutative diagram:
        \[
        \begin{CD}
        f_*\omega_Y @> \sim >> f_*\sO_Y(K_Y) @. \subseteq f_*K(Y) \\
        @V \Tr_f VV @. @| .\\
        \omega_X @ > \sim>> \sO_X(K_X) @. \subseteq K(X) \\
        \end{CD}
        \]
        \item Assume that $f$ is finite and surjective.
        Then $\omega_Y=f^{-1}\SHom_X(f_*\sO_Y, \omega_X)$ and $\Tr_f$ is the 
evaluation at $1$ map.
        Furthermore, there exists an isomorphism $f_* \sO_Y \xrightarrow{\sim} \SHom_X(f_* \omega_Y, \omega_X)$ which sends $1$ to $\Tr_f$.
        \item Assume that $f$ is finite and generically separable.
        We can define Weil divisors $\Ram_f$ and $f^*(K_X)$ on $Y$ as in \cite[Definition 4.5]{ST1}. 
        Then $K_Y=f^*(K_X)+\Ram_f$ is a canonical divisor on $Y$ such that $\sO_Y(K_Y) \simeq \omega_Y$ and there exists a commutative diagram
        \[
	\begin{CD}
        f_*\omega_Y @> \sim >> f_*\sO_Y(K_Y) @. \subseteq f_*K(Y) \\
        @V \Tr_f VV @. @V \Tr_{K(Y)/K(X)}VV ,\\
        \omega_X @ > \sim >> \sO_X(K_X) @. \subseteq K(X) \\
        \end{CD}
        \]
        where $\Tr_{K(Y)/K(X)}$ is the trace map of the field extension.
\end{enumerate}
\end{eg}

A scheme $X$ of positive characteristic is called \emph{$F$-finite} if the Frobenius morphism $F:X \to X$ is a finite morphism.
A ring $R$ is called \emph{$F$-finite} if $\Spec R$ is $F$-finite.
An $F$-finite Noetherian affine scheme is excellent and finite dimensional (\cite{Kun}), and has a canonical module (\cite[Remark13.6]{Gab}).
If $X$ is F-finite and $F^!\omega_X \cong \omega_X$, then we can define $\Tr_{F^e_X}:F^e_*\omega_X \to \omega_X$ for every $e \ge 1$.
In this case, if we fix a canonical divisor $K_X$ on $X$, there exists a $K(X)$-homomorphism $\rho: F^e_*K(X) \to K(X)$ such that
\[
\begin{CD}
        F^e_*\omega_X @>\sim>> F^e_*\sO_X(K_X) @. \subseteq F^e_*K(X) \\
        @V \Tr_{F^e} VV @. @V \rho VV .\\
        \omega_X @ >\sim>> \sO_X(K_X) @. \subseteq K(X) \\
        \end{CD}
\]
We denote this $\rho$ by $\Tr_{F^e_X,K_X}$.
When $X=\Spec R$, we also denote this $\rho$ by $\Tr_{F^e_R, K_R}$.

\begin{rem}\hfill
\begin{enumerate}
	\item If $X$ is of finite type over the spectrum of a local ring, then the condition $F^!\omega_X \cong \omega_X$ holds (see \cite[Example 2.15]{BST}).
        \item Although $\Tr_{F^e_X, K_X}: F^e_*K(X) \to K(X)$ is not uniquely determined, it is unique up to multiplication by $H^0(X,F^e_*\sO_X^\times)$.
        When we use this map, we only need the information about the image of this map as in Proposition \ref{phi principal} below.
        Hence we ignore the multiplication by $H^0(X, F^e_*\sO_X^\times)$.
\end{enumerate}
\end{rem}

\begin{prop}[\textup{\cite[Corollary 5.5]{ST1}}]\label{Tr on fin}
	Suppose that $X$ and $Y$ are Noetherian normal connected $F$-finite schemes with $\omega_X \cong F^! \omega_X$ and $\omega_Y \cong F^! \omega_Y$.
        Let $f:Y \to X$ be a finite dominant generically separable morphism, and $K_X,K_Y$ be canonical modules as in \textup{Example \ref{tr}} $(3)$.
        Then for every $e>0$, we have the following commutative diagram:
	\[
	\begin{CD}
        F^e_*K(X) @>\subseteq>> F^e_*K(Y) \\
        @V \Tr_{F^e, K_X} VV  @VV \Tr_{F^e, K_Y} . V\\
        K(X) @ >\subseteq>> K(Y) \\
        \end{CD}
	\]
\end{prop}

\subsection{Test ideals}\label{test ideal}
In this subsection, we assume all rings and schemes are characteristic $p>0$ and $F$-finite.
We define test ideals and recall some properties of test ideals.

\begin{defn}
	A \emph{pair} $(X, \D)$ consists of a Noetherian normal connected scheme $X$ and an $\R$-Weil divisor $\D$ on $X$.
        A \emph{triple} $(X, \D, \fa^t)$ consists of a  pair $(X, \D)$ and a symbol $\fa^t$, where $\fa$ is a coherent ideal sheaf and $t \in \R_{\ge 0}$.
        When additionally $K_X+\D$ is a $\Q$-Cartier $\Q$-divisor, $(X,\D)$ (resp. $(X, \D, \fa^t)$) is called a log $\Q$-Gorenstein pair (resp. triple).
\end{defn}

\begin{defn}
	Suppose that $(X=\Spec R, \D, \fa^t)$ is a triple with $\D$ effective.
        An ideal $J \subseteq R$ is \emph{uniformly $(\D, \fa^t, F)$-compatible} if 
        $\phi(F^e_*(\fa^{\lceil t(p^e-1) \rceil}J))\subseteq J$ for every $e \ge 1$ and every $\phi \in \Hom_R(F^e_*R(\lceil (p^e-1)\D \rceil ),R)$.
\end{defn}

\begin{defn}
	Let $(X= \Spec R, \D, \fa^t)$ be a triple.
        Assume that $\D$ is effective and $\fa \neq (0)$.
        We define the \emph{test ideal} $\tau(R, \D, \fa^t)=\tau(X, \D, \fa^t)$ to be the unique minimal non-zero uniformly $F$-compatible ideal.
        This ideal always exists (see \cite{Sch1}).
        
        When $\fa=R$, we denote $\tau(R, \D, \fa^t)$ by $\tau(X, \D)$.
        
        When $\fa=(0)$, we define $\tau(X, \D, \fa^t) \coloneqq (0)$.
\end{defn}

We recall some properties of test ideals.

\begin{prop}[\textup{basic properties}]\label{test basic}
	Let $(X, \D, \fa^t)$ be a triple.
        Assume that $X= \Spec R$ is affine and $\D$ is effective.
        Then the following properties hold.
	\begin{enumerate}
	\item \textup{\cite[Proposition 3.1]{HT}} 
        For every multiplicative closed subset $T \subseteq R$, we have 
$\tau(R, \D, \fa^t)T^{-1}R=\tau(T^{-1}R, \D|_{T^{-1}R}, T^{-1}\fa^t)$.
        \item \textup{\cite[Proposition 3.2]{HT}} 
        If $(R, \m)$ is local and $\widehat{R}$ is the $\m$-adic completion, then $\tau(R, \D, \fa^t) \cdot \widehat{R} = \tau(\widehat{R}, \widehat{\D}, (\fa \cdot \widehat{R})^t)$,
         where $\widehat{\D}$ is the flat pullback of $\D$ by $\Spec \widehat{R} \to X$.
	\item \textup{\cite[p.9 Basic property]{Tak}}
        For every Cartier divisor $D$ on $X$, $\tau(X, \D+D, \fa^t)=\tau(X, \D, \fa^t) \cdot \sO_X(-D) $.
        \item \textup{\cite[Lemma 3.1]{Sch3}}
Assume that $\fa= f \cdot R$  for some $f \neq 0$.
        Then we have $\tau(X, \D, \fa^t)= \tau( X, \D + t\Div(f))$.
        \end{enumerate}
\end{prop}

Let $(X, \D, \fa^t)$ be a triple.
By Proposition \ref{test basic} (1) and (3), we can define a test ideal $\tau(X, \D, \fa^t)$.
This is a coherent fractional ideal.

\begin{defn}\label{st. F-reg}
Let $(X, \D)$ be a pair with $\D$ effective and $P$ be a point of $X$.
$(X, \D)$ is \emph{strongly $F$-regular} at $P$ if $\tau(X, \D)_P=\sO_{X, P}$.
$(X, \D)$ is called strongly $F$-regular if $(X, \D)$ is strongly $F$-regular at every point $P$ in $X$.
When $(X, 0)$ is strongly $F$-regular, we say simply that $X$ is strongly $F$-regular.
\end{defn}

\begin{prop}[\textup{\cite[Theorem 5.1]{Sch2} and \cite[Lemma 2.5]{Sch3}}]\label{non gor test}
	Let $(X =\Spec R, \D, \fa^t)$ be a triple with $\D$ effective.
        Then we have
        \[
        \tau(X, \D, \fa^t)=\sum_{\D'} \tau(X, \D', \fa^t),
        \]
        where $\D'$ runs through all $\Q$-Weil divisors on $X$ such that $\D' \ge \D$ and $(p^e-1)(K_X+\D')$ is an integral Cartier divisor for some $e>0$.
\end{prop}


\begin{defn}\label{phi}
Let $(X, \D)$ be a pair with $\D$ effective, $\omega_X=\sO_X(K_X)$ be a canonical module on $X$ such that $F^!\omega_X=\omega_X$, and $e$ be a positive integer.
\begin{enumerate}
\item If $(p^e-1)(K_X+\D)$ is an integral Cartier divisor, we denote the invertible sheaf $\sO_X((1-p^e)(K_X+\D))$ by $\sL_{e,\D}$ and we define
\[
\widetilde{\phi}_{e,\D} : F^e_*\sL_{e, \D} \stackrel{\subseteq}{\longrightarrow} F^e_*\sO_X((1-p^e)K_X) \stackrel{\Tr_{F^e_X, K_X}}{\longrightarrow} \sO_X.
\]
\item If $(p^e-1)(K_X+\D)$ is a principal divisor, we define
\[
\phi_{e, \D}: F^e_*\sO_X \stackrel{\simeq}{\longrightarrow} F^e_*\sL_{e, \D} \stackrel{\widetilde{\phi}_{e,\D}}{\longrightarrow} \sO_X.
\]
\end{enumerate}
\end{defn}

\begin{prop}[\textup{\cite[Lemma 3.21]{BSTZ}}]\label{phi acc}
	Let $(X, \D, \fa^t)$ be a triple and $q=p^e$ for some positive integer $e$.
        Assume that $X$, $\D$ and $e$ are as in \textup{Definition \ref{phi} (1)}.
        Then the sequence $\{ \widetilde{\phi}_{en, \D} (F^{en}_*(\fa ^{\lceil q^nt \rceil} \tau(X, \D)\sL_{en, \D}) )\}_n$ 
        is an ascending chain of coherent ideals which converges to $\tau(X, \D, \fa^t)$.
\end{prop}

\begin{prop}[\textup{\cite[Lemma 5.4]{BSTZ}}]\label{phi principal}
	Let $(X=\Spec R, \D)$ be a pair and $q=p^e$ for some positive integer $e>0$.
        Assume that $X$, $\D$ and $e$ are as in \textup{Definition \ref{phi} (2)}.
        Then, for every non-zero element $f \in R$, we have
        \[
        \phi_{e, \D}(f \cdot \tau(X, \D))= \tau(X, \D+ \frac{\Div(f)}{q}).
        \]
\end{prop}

\begin{prop}[\textup{\cite[Lemma 3.23]{BSTZ}}]\label{fpt}
	Suppose that $(X, \D, \fa^t)$ is a triple.
        Then the following properties hold.
        \begin{enumerate}
        \item \textup{\cite[Lemma 6.1]{ST2}}
        We have $\tau (X, \D, \fa^{t+\epsilon}) \subseteq \tau(X, \D, \fa^t)$ for all $0 \le \epsilon$,
        and \\
$\tau (X, \D, \fa^{t+\epsilon}) = \tau(X, \D, \fa^t)$ for all $0 \le \epsilon \ll 1$.
        \item Let $D$ be an effective Cartier divisor on $X$.
        Then $\tau(X, \D+\epsilon D, \fa^t) \subseteq \tau(X, \D, \fa^t)$ for all $0 \le \epsilon$,
        and $\tau(X, \D+\epsilon D, \fa^t)=\tau(X, \D , \fa^t)$ for all $0 \le \forall \epsilon \ll 1$.
        \end{enumerate}
\end{prop}

\begin{proof}
        For (2), we may assume that $X=\Spec R$ and $D = \Div(f)$ for some $f \in R$.
        By \cite[Lemma 3.1]{Sch3}, we have $\tau(X, \D+ \epsilon D, \fa^t)=\tau(X, \D, \fa^t \fb^\epsilon)$, where $\fb \coloneqq f \cdot R \subseteq R$.
        We need to show $\tau(X, \D, \fa^t)= \tau(X, \D, \fa^t \fb^\epsilon)$ for all $0 < \epsilon \ll 1$.
        By Proposition \ref{non gor test}, we may assume that $(p^e-1)(K_X+ \D)$ is an integral divisor for some $e>0$.
        In this case, the proof is similar to the proof of \cite[Lemma 3.23]{BSTZ}.
\end{proof}

Next, we introduce the asymptotic version of test ideals.

\begin{defn}
Let $X$ be a scheme. A sequence of coherent ideals $\fa_\bullet = \{ \fa_n \}_{n \ge 1}$ is called a \emph{graded sequence of ideals} if $ \fa_n \cdot \fa_m \subseteq \fa_{m+n}$ for every $n,m \ge 1$.
A graded sequence of ideals $\fa_\bullet$ is called \emph{non-zero} if $\fa_n \neq0$ for some $n \ge 1$.
\end{defn}

\begin{propdef}[\textup{\cite[p.8]{Mus}, cf. \cite[Definition 11.1.15]{Laz}}]\label{asympt}
	Let $(X,\D)$ be a pair, $t$ be an element in $\R_{ \ge0}$, and $\fa_\bullet$ be a graded sequence of ideals.
        Then the set of coherent fractional ideals $\{\tau(X, \D, \fa_n^{t/n}) \}_{n\ge 1}$ has a unique maximal element.
        We denote this coherent fractional ideal by $\tau(X, \D, \fa_\bullet^t)$ and call it the \emph{asymptotic test ideal}.
        Furthermore, $\tau(X, \D, \fa_\bullet^t)=\tau(X, \D, \fa_n^{t/n})$ for all divisible $n$.
\end{propdef}
\begin{proof}
	The proof is similar to the proof of \cite[Definition 11.1.15]{Laz}.
\end{proof}

\begin{rem}\hfill
	\begin{enumerate}
        \item $\tau(X, \D, \fa_\bullet^t) \neq 0 \Leftrightarrow \fa_\bullet$ is non-zero.
        \item In general, the above definition differs from the one given in \cite{Hara}. See \cite[Remark 1.4]{TY}.
        \item When $X$ is regular and $\D=0$, this definition coincides with the one given in \cite{Mus}.
        \end{enumerate}
\end{rem}

\subsection{Multiplier ideals}\label{multiplier ideal}
Assume that $f \colon Y \to X$ is a birational proper morphism of finite dimensional Noetherian normal connected schemes and $K_X, K_Y$ be as in Example \ref{tr} (1).
If $\D$ is a $\Q$-Weil divisor on $X$ such that $K_X+\D$ is a $\Q$-Cartier $\Q$-divisor, we can define $\D_Y=-K_Y+f^*(K_X+\D)$.

Let $E \subseteq Y$ be a prime divisor on $Y$. A discrete valuation $v=\ord_E: K(X) \to \Z \cup \{ \infty \}$ is called the \emph{divisorial valuation} over $X$, and $\cen_X(E)=\cen_X(v) \coloneqq f(E)\subseteq X$ is called the \emph{center} of $E$ or $v$.

\begin{defn}[\textup{\cite[Definition 2.18]{BST}}]\label{multiplier}
Let $(X, \D, \fa^t)$ be a log $\Q$-Gorenstein triple.
We define the \emph{multiplier ideal} by
\[
\J(X, \D, \fa^t) \coloneqq \bigcap_{f:Y \to X} f_*\sO_Y(- \lfloor \D_Y+tF \rfloor),
\]
where $f:Y\to X$ runs through all proper birational morphisms such that $\fa \cdot \sO_Y=\sO_Y(-F)$ for some effective divisor $F$ on $Y$.

When $\fa=\sO_X$, we denote $\J(X, \D, \fa^t)$ by $\J(X, \D)$.
When $\fa =0$, we define $\J(X, \D, \fa^t)=(0)$.
\end{defn}

\begin{defn}\label{asy mult}
	Let $X$, $\D$ and $t$ be as in Definition \ref{multiplier} and $\fa_\bullet$ be a graded sequence of ideals.
Then we define the asymptotic multiplier ideal $\J(X ,\D ,\fa^t_\bullet)$ by
\[
\J(X, \D, \fa^t_\bullet) : = \sum_{n \ge 1} \J(X, \D, \fa^{t/n}_n)
\] 
\end{defn}

\begin{rem}
If $X$ is the spectrum of a local ring or $X$ is a variety over a field of characteristic $0$, then $\J(X, \D, \fa_n^{t/n})$ is coherent for every $n$.
In this case, the set of coherent fractional ideal sheaves $\{ \J(X, \D, \fa_n^{t/n}) \}_{n\ge 1}$ has a unique maximal element, which corresponds to $\J(X, \D, \fa^t_\bullet)$.
\end{rem}

\subsection{Divisorial valuations and orders}\label{divisorial valuation}

In this subsection, we recall the definitions and some properties of divisorial valuations and orders.
\begin{defn}
Let $(R, \m)$ be a Noetherian local ring, $f$ be an element in $R$, $\fa$ be an ideal, and $\fa_\bullet$ be a graded sequence of ideals.
\begin{enumerate}
\item We define $\ord_R(f) \coloneqq \sup\{ r\ge 0 \mid f \in \m^r\} \in \N \cup \{ \infty \}$.
\item We define $\ord_R(\fa) \coloneqq \sup\{ r\ge 0 \mid \fa \subseteq \m^r\} =\inf \left\{ \ord_R(f) \mid f \in \fa \right\} \in \N \cup \{ \infty \}$.
\item We define $\ord_R(\fa_\bullet) \coloneqq \inf_n \ord_R( \fa _n)/n \in \R_{\ge 0} \cup \{ \infty \}$.
\end{enumerate}
\end{defn}

\begin{defn}
Let $(R, \m)$ be a Noetherian local domain with $K=\Frac(R)$ and $k=R/\m$.
A discrete valuation $v: K \to \Z \cup \{ \infty \}$ is called an $\m$-valuation 
if $v$ is non-negative on $R$, positive on $\m$, and $\mathrm{tr.deg}_k k(v)=\dim R-1$ where $k(v)$ is the residue field of the discrete valuation ring associated to $v$.
\end{defn}

\begin{rem}
Let $X$ be as in Definition \ref{multiplier}, $v$ be a discrete valuation of $K(X)$, and $P$ be a point of $X$.
Then the following are equivalent:
\begin{enumerate}
\item $v$ is a divisorial valuation such that $\cen_X(v)=\overline{ \{ P \}}$.
\item $v$ is an $\m_P$-valuation of $(\sO_{X,P}, \m_P)$.
\end{enumerate}
(See \cite[Lemma 2.45]{KM}).
\end{rem}

\begin{defn}
Let $(R, \m)$ be a Noetherian local ring, $v$ be an $\m$-valuation, $\fa$ be an ideal, and $\fa_\bullet$ be a graded sequence of ideals.
\begin{enumerate}
\item We define $v(\fa) \coloneqq \inf \left\{ v(f) \mid f \in \fa \right\} \in \N \cup \{ \infty \}$.
\item We define $v(\fa_\bullet) \coloneqq \inf_n v( \fa _n)/n \in \R_{\ge 0} \cup \{ \infty \}$.
\end{enumerate}
\end{defn}

Let $X$ be a normal Noetherian scheme, $P$ be a point of $X$ and $v$ be a divisorial valuation such that $\cen_X(v)=\overline{\{ P \} }$.
For every effective Cartier divisor $D$ on $X$, we define 
$\ord_P(D) \coloneqq \ord_R(f_D)$ and $v(D) \coloneqq v(f_D)$, where $R=\sO_{X,P}$ and $f_D \in R$ is the equation of $D$ at $P$.
For a coherent ideal $\fa$, we define $\ord_P(\fa) \coloneqq \ord_R(\fa_P)$ and $v(\fa)=v(\fa_P)$.
In the same way, we can define $\ord_P(\fa_\bullet)$ and $v(\fa_\bullet)$.

\begin{prop}[\textup{Izumi's theorem cf.\cite{Ree}}]\label{Izumi}
	Let $(R, \m)$ be a Noetherian local ring such that $\widehat{R}$ is an integral domain.
        (This condition holds when $R$ is normal and excellent.)
        Then the following properties hold.
	\begin{enumerate}
	\item For every $\m$-valuations $v$ and $w$, there exists $C>0$ such that $v(f) \le C \cdot w(f)$ for all $f \in R$.
        \item For every $\m$-valuation $v$, there exists $C>0$ such that $\ord_R(f) \le v(f) \le C \cdot \ord_R(f)$ for all $f \in R$.
        \end{enumerate}
\end{prop}

Finally, as an application of Izumi's theorem, we obtain the lemma below, which is used in the proof of the main theorem.
Let $A \subseteq R$ be a finite extension of normal Noetherian excellent local domains and $K=\Frac(A), L=\Frac(R)$.
We consider $\rN=\rN_{L/K}:L^\times \to K^\times$.

\begin{lem}\label{norm}
	With the above notation, there exists a constant $M>0$ such that $\ord_A( \rN(r)) < M \cdot \ord_R(r)$ for all $r \neq 0 \in R$.
\end{lem}

\begin{proof}
	Let $w$ be an $\m_A$-valuation and $\sO_w$ be the valuation ring associated to $w$.
        The integral closure $B=\overline{\sO_w}^L$ is a semilocal ring.
        Let $\sO_{v_1}, \dots , \sO_{v_n}$ be the localizations of $B$ by maximal ideals and $v_i$ be the valuation associated to $\sO_{v_i}$ for every $i$.
        Then $v_i$ is an $\m_R$-valuation for every $i$.
        By \cite[Lemma A.3]{Ful}, we have
        \[ w(\rN(r))=\sum_i v_i(r)\cdot [k(w):k(v_i)]\]
        for every $r \in R$.
	By Proposition \ref{Izumi} (2), there exists a constant $C_i>0$ such that $v_i(r) \le C_i \cdot \ord_R(r)$ for every $i$ and every $r \in R$.
When we put $M :=\sum_i C_i \cdot [k(w): k(v_i)]$, we have $\ord_A(\rN(r)) \le w(\rN(r)) \le M \cdot \ord_R(r)$ for every non-zero element $r \in R$.
\end{proof}

\subsection{Invariants on projective varieties}\label{projective variety}

In this subsection, we assume that $X$ is a normal projective variety over an $F$-finite field.

\begin{defn}
For an $\R$-Cartier $\R$-divisor $D$ on $X$, we define
\begin{eqnarray*}
|D|_{\Q} & \coloneqq & \left\{ E \ge 0 \mid E \sim_{\Q} D \right\} \hspace{0.5cm} \textup{and} \\
|D|_{\equiv} & \coloneqq & \left\{ E \ge 0 \mid E \equiv_{\num} D \right\} .
\end{eqnarray*}
\end{defn}

\begin{defn}
For a $\Q$-Cartier $\Q$-divisor $D$ on $X$, we define the \emph{stable base locus} of $D$ as
\[
\B(D) \coloneqq \bigcap_{E \in |D|_{\Q}} \Supp (E).
\]
For an $\R$-Cartier $\R$-divisor $D$ on $X$, we define the \emph{restricted base locus} of $D$ as
\[
\B_{-}(D) \coloneqq \bigcup_{A} \B(D+A),
\]
where $A$ runs through all ample $\R$-Cartier $\R$-divisors on $X$ such that $A+D$ is a $\Q$-Cartier $\Q$-divisor.
We remark that when $D$ is not pseudo-effective, we put $\B_{-}(D)=X$.
\end{defn}

\begin{defn}
Let $D$ be a $\Q$-Cartier $\Q$-divisor on $X$ and $v$ be a divisorial valuation.
We define the \emph{asymptotic order of vanishing} of $D$ along $v$ as
\[v(||D||)  \coloneqq  \inf _{E \in |D|_{\Q}} v(E) \in \R_{\ge 0} \cup \{ \infty \}. \]
\end{defn}

\begin{rem}\label{||D||}
For a $\Q$-Cartier $\Q$-divisor $D$ on $X$, we can define a graded sequence of ideals $\fa_\bullet (D)$ as 
\[
\fa_n (D)= \left\{ \begin{array}{ll} \fb(|nD|) & \textup{(if $nD$ is integral)} \\ 0 & \textup{(otherwise)} \end{array} \right. .
\]
Then we have $v(||D||)=v(\fa_\bullet(D))$.
\end{rem}

\begin{defn}[\textup{\cite[Definition 2.10]{CdB}}]\label{v_{num}}
Let $v$ be a divisorial valuation.
For every big $\R$-Cartier $\R$-divisor $D$ on $X$, we define the \emph{numerical vanishing order} of $D$ along $v$ as 
\[v_{\num}(D) \coloneqq  \inf _{E \in |D|_{\num}} v(E) \in \R_{\ge 0}.\]

For every pseudo-effective and non-big $\R$-Cartier $\R$-divisor $D$ on $X$, we define
\[ v_{\num}(D) \coloneqq \sup_A v_{\num}(D+A),\]
where $A$ runs through all ample $\R$-Cartier $\R$-divisors on $X$.
\end{defn}

\begin{prop}[\textup{\cite[Theorem A and Proposition 2.4]{ELMNP}}]\label{vnum basic}
Let $v$ be a divisorial valuation.
\begin{enumerate}
\item If $D$ is a big $\Q$-Cartier $\Q$-divisor, we have $v(||D||)=v_{\num}(D)$.
\item Let $\mathrm{Big}(X)_\R \subseteq \rN^1(X)_\R$ be the cone of numerical equivalence classes of big $\R$-divisors on $X$.
Then the map $v_{\num}:\mathrm{Big}(X)_\R \to  \R$ is continuous.
\item If $D$ and $E$ are big $\Q$-Cartier $\Q$-divisors, then $v_{\num}(D+E) \le v_\num(D) + v_\num(E)$
\end{enumerate}
\end{prop}

The following follows from Proposition \ref{vnum basic}

\begin{cor}
Let $v$ be a divisorial valuation, $D$ be a pseudo-effective $\R$-Cartier $\R$-divisor, and  $(A_m)_{m\in \N}$ be a sequence of ample $\R$-Cartier $\R$-divisor divisors on $X$ such that $\lim_{m \to \infty}A_m=0 \in \rN^1(X)$.
Then $v_{\num}(D)=\lim_{m\to \infty}v_{\num}(D+A_m)$
\end{cor}

\begin{defn}\label{NNef}
For every $\Q$-Cartier $\Q$-divisor $D$ on $X$ such that $|D|_{\Q} \neq \emptyset$, we define
\[
\NNA(D) \coloneqq \bigcup_v \cen_X(v),
\]
where $v$ runs through all divisorial valuations such that $v(||D||)>0$.
When $|D|_{\Q}=\emptyset$, we define $\NNA(D)=X$.

For every pseudo-effective $\R$-Cartier $\R$-divisor $D$ on $X$, we define the \emph{non-nef locus} of $D$ as
\[
\NNef(D) \coloneqq \bigcup_v \cen_X(v)
\]
where $v$ runs through all divisorial valuations such that $v_{\num}(D)>0$.
When $D$ is not pseudo-effective, we define $\NNef(D)=X$.
\end{defn}

\begin{rem}
	For every $\R$-Cartier $\R$-divisor $D$ on $X$, $D$ is nef if and only if $\NNef(D)$\\
$= \emptyset$ (\cite[III, Lemma 2.7]{Nak}).
        In characteristic 0, $D$ is nef and abundant if and only if $\NNA(D) = \emptyset$ (\cite[Lemma 2.17]{CdB}) and hence $\NNA(D)$ is denoted by $\mathrm{NNA}(D)$ and called the \emph{non nef-abundant locus}.
        However, because we do not know whether this is true in positive characteristic, we avoid using the notation $\mathrm{NNA}(D).$ 
\end{rem}

\begin{defn}
Let $\D$ be an $\R$-Weil divisor on $X$, $D$ be a $\Q$-Cartier $\Q$-divisor on $X$, and $t$ be a positive number.
We define 
\[\tau(X, \D, t||D||) \coloneqq \tau(X, \D, \fa_{\bullet}(D)^t),\]
where $\fa_\bullet(D)$ is defined in Remark \ref{||D||}.
\end{defn}

\section{A proof of the main theorem}

In this section, we give a description of the trace map for the Frobenius morphism on a complete local ring (Proposition \ref{decompo R}).
By using this description, we prove the main theorem (Corollary \ref{main'''}).

\subsection{The trace map on the ring of formal power series}

Let $k$ be an $F$-finite field of characteristic $p>0$, let $d$ be a positive integer, let $A$ be the ring of formal power series $k[[x_1, \dots , x_d]]$, and let $K$ be the field of fractions of $A$.
Since $A$ is a regular local ring, we may assume that the canonical divisor $K_A$ on $A$ is trivial.
We want to consider $\Tr_{F^e_A, K_A}:F^e_*K \to K$.

For $\lambda=(\lambda_1, \dots , \lambda_d) \in \N^d$, we write $x^\lambda \coloneqq x_1^{\lambda_1}x_2^{\lambda_2} \cdots x_d^{\lambda_d} \in A$
and $\deg \lambda \coloneqq \sum_{i=1}^d \lambda_i$.

\begin{defn}\label{junjo}
We define the \emph{degree-lexicographic order} $\le$ in $\N^d$ as
$\lambda \le \mu$ if and only if $\deg \lambda < \deg \mu$, or
$\deg \lambda = \deg \mu$ and $\lambda \le \mu$ in the lexicographic order.
By this order, $\N^d$ becomes a totally ordered set.
\end{defn}

\begin{defn}
For every positive integer $e$, we denote 
\[I_e \coloneqq \left\{ \lambda =(\lambda_1, \dots ,\lambda_d) \in \N^d \mid  0 \le \lambda_i < p^e \textup{ for all } i \right\}.\]
This $I_e$ is also a totally ordered set by the degree-lexicographic order.

We denote $\mu_e \coloneqq (p^e-1, \dots , p^e-1) \in I_e$.
\end{defn}

\begin{defn}\label{dec}
	With the above notation, we write $A^{p^e} \cdot k$ for the compositum of rings $A^{p^e}$ and $k$, that is, 
        \[
        A^{p^e} \cdot k=\{ \sum_i a_i^{p^e} \cdot u_i \mid a_i \in A, u_i \in k \} \subseteq A.
        \]
        Similarly, we write $K^{p^e} \cdot k \subseteq K$ for the compositum of fields $K^{p^e}$ and $k$.
\end{defn}

\begin{lem}\label{decompo A}
	With the above notation, the following properties hold.
        \begin{enumerate}
        \item We have $A= \bigoplus_{\lambda \in I_e} (A^{p^e} \cdot k)x^\lambda$ and
        	$K= \bigoplus_{\lambda \in I_e} (K^{p^e} \cdot k)x^\lambda$.
        \item Let $\vt =\{ t_1, t_2, \dots , t_{l} \} $ be a basis of $k$ over $k^{p^e}$.
        	Then, 
                \[A^{p^e} \cdot k = \bigoplus_{1 \le i \le l} A^{p^e} \cdot t_i \textup{ and }
                K^{p^e} \cdot k = \bigoplus_{1 \le i \le l} K^{p^e} \cdot t_i.\]
        \item Let $\vt$ be as above. Then, 
                \[
                A=\bigoplus_{\lambda \in I_e , 1 \le i \le l} A^{p^e} \cdot (t_i x^\lambda) \textup{ and } 
                K=\bigoplus_{\lambda \in I_e , 1 \le i \le l} K^{p^e} \cdot (t_i x^\lambda).
                \]
                In other words, 
                \[
                F^e_*A=\bigoplus_{\lambda \in I_e , 1 \le i \le l} A \cdot F^e_*(t_i x^\lambda) \textup{ and } 
                F^e_*K=\bigoplus_{\lambda \in I_e , 1 \le i \le l} K \cdot F^e_*(t_i x^\lambda),
                \]
where for any $f \in A$, we denote by $F^e_*(f)$ the element of $F^e_*A$ which corresponds to $f$ by the canonical group isomorphism $A \xrightarrow{\sim} F^e_*A$.
        \item Let $\vt$ be as above. For every $1 \le j \le l$, the projection
		\[
        	p_{\vt,(\mu_e,j)}:F^e_*A=\bigoplus _{\lambda \in I_e, 1 \le i \le l} A\cdot F^e_*(t_ix^\lambda) \to A \cdot F^e_*(t_j x^{\mu_e})=A
        	\]
        	satisfies $\Hom_A(F^e_*A, A) =F^e_*A \cdot p_{\vt,(\mu_e, j)}$.
        \item Let $\vt$ be as above. For every $1 \le j \le l$, we have $\Tr_{F^e_A, K_A} = p_{\vt, (\mu_e ,j)}: F^e_* K \to K$.
        \end{enumerate}
\end{lem}

\begin{proof}
	(1), (2), (3), and (4) are easy.
        (5) follows from (4) and Example \ref{tr} (2).
\end{proof}

\begin{defn}
        For every ideal $J \subseteq A$, we define
        \[
        J^{[p^e]} \cdot k=\{ \sum_i a_i^{p^e} \cdot u_i \mid a_i \in J, u_i \in k \} \subseteq A^{p^e} \cdot k.
	\]
\end{defn}

\begin{lem}\label{compose ideal}
With the above notation, the following properties hold.
        \begin{enumerate}
        \item $J^{[p^e]} \cdot k \subseteq A^{p^e} \cdot k$ is an ideal.
        \item Let $\vt =\{ t_1, t_2, \dots , t_{l} \} $ be a basis of $k$ over $k^{p^e}$ and $\alpha=\sum_{1 \le i \le l} a_i^{p^e}t_i \in A^{p^e} \cdot k$.
	Then $\alpha \in J^{[p^e]} \cdot k$ if and only if $a_i \in J$ for all $i$.
        \item If $J$ is the maximal ideal $\m_A$ of $A$, the assertion in $(2)$ is equivalent to $\alpha \in \m_A$.
        \end{enumerate}
\end{lem}

\begin{proof}
	(1) is trivial. 
        For (2), we assume that $\alpha \in J^{[p^e]} \cdot k$.
        We can write $\alpha= \sum_{j} r_j^{p^e} u_j$ for some $r_j \in J$ and $u_j \in k$.
        Furthermore, we can write $u_j=\sum_{1 \le i \le l}u_{i,j}^{p^e} t_i$ for some $u_{i,j} \in k$.
        Then we have $a_i=\sum_{j} u_{i,j}r_j$. Hence $a_i \in J$.
        For (3), we assume that $\alpha \in \m_A$.
        Then we have $0=\sum_{i} \overline{a_i}^{p^e} t_i \in A / \m_A = k$.
        Hence, $\overline{a_i}=0$ for all $i$.
\end{proof}

\subsection{Trace maps on complete local rings}
Let $(R, \m)$ be an F-finite Noetherian complete normal local ring of characteristic $p>0$ with $\dim R=d>0$ and $k=R/\m$ be the residue field.

\begin{prop}[\cite{KS}, Cohen-Gabber Theorem]\label{Cohen}
Let $A= k[[x_1, x_2, \dots , x_d]]$ be the ring of formal power series.
Then there exists a coefficient field $k \subseteq R$ and a system of parameters $y_1, \dots, y_d \in R$ with the following property:
if $A \subseteq R$ is the module finite inclusion defined by the inclusion $k \subseteq R$ and the system of parameter $y_1, \dots, y_d$, then the field extension $\Frac(A) \subseteq \Frac(R)$ defined by $A \subseteq R$ is separable.
\end{prop}

Fix inclusions $k \subseteq R$ and $A \subseteq R$ as in Proposition \ref{Cohen}.
Then we can define $R^{p^e} \cdot k$ and $L^{p^e} \cdot k$ as in Definition \ref{dec}, where $L \coloneqq \Frac (R)$.
Similarly, we can define $J^{[p^e]} \cdot k \subseteq R^{p^e} \cdot k$ for every ideal $J \subseteq R$.
We can take $K_A=0$ and $K_R=\Ram_\pi$, where $\pi: \Spec R \to \Spec A$ is the finite morphism induced by $A \subseteq R$.
\begin{prop}\label{decompo R}
	With the above notation, the following properties hold.
	\begin{enumerate}
        \item Let $\vt =\{ t_1, t_2, \dots , t_{l} \} $ be a basis of $k$ over $k^{p^e}$.
        Then we have 
        \[
        L=\bigoplus_{\lambda \in I_e} (L^{p^e} \cdot k) x^\lambda=\bigoplus_{\lambda \in I_e, i}L^{p^e} \cdot (t_ix^\lambda).
        \]
        \item For every $1 \le j \le l$, we have $\Tr_{F^e_R, K_R}=p_{\vt, (\mu_e, j)};F^e_* L \to L$.
        \end{enumerate}
\end{prop}

\begin{proof}
Since $F^e_*L= L \otimes_K F^e_* K$ (cf. \cite[Lemma 3.3]{ST1}), we have (1). 
By Proposition \ref{Tr on fin}, we have (2).
\end{proof}

\begin{lem}\label{compose ideal 2}
	Let $J$ be an ideal of $R$.
        Then the following properties hold.
        \begin{enumerate}
        \item $J^{[p^e]} \cdot k \subseteq R^{p^e} \cdot k$ is an ideal.
        \item  Let $\vt =\{ t_1, t_2, \dots , t_{l} \} $ be a basis of $k$ over $k^{p^e}$ and $\alpha=\sum_{1 \le i \le l} a_i^{p^e}t_i \in L^{p^e} \cdot k$.
	Then $\alpha \in J^{[p^e]} \cdot k$ if and only if $a_i \in J$ for all $i$.
        \end{enumerate}
\end{lem}
\begin{proof}
	The proof is similar to the proof of Lemma \ref{compose ideal}.
\end{proof}

\subsection{A proof of the main theorem}

\begin{thm}\label{main}
	Let $(R, \m)$ be an F-finite Noetherian normal local domain of characteristic $p>0$ and $d=\dim R>0$.
	Assume that $\D$ is an effective $\Q$-Weil divisor on $R$ such that $K_R+\D$ is $\Q$-Cartier $\Q$-divisor.
	Then there exist positive integers $M$ and $q=p^{e}$ such that for all $n\ge 1$ and for all $r \neq 0 \in R$ with $\ord_R(r)<q^{n-1}/M$, we have 
        \[\tau(R,\D)=\tau ( R, \D+\frac{ \Div (r)}{q^n} ) .\]
\end{thm}

By Proposition \ref{test basic} (2), it is enough to show the theorem when $R$ is a complete local ring.
Before proving the theorem, we fix some notations.

\begin{setup}\label{3.12}
Let $(R, \m)$ be an $F$-finite complete Noetherian local normal domain of characteristic $p>0$, $k:=R/ \m$ be the residue field, $d=\dim(R)>0$, and $\D$ be an effective $\Q$-Weil divisor such that $K_R+ \D$ is a $\Q$-Cartier $\Q$-divisor.
We define $A := k[[x_1, x_2, \dots , x_d]]$ and fix inclusions $k \subseteq R$ and $A \subseteq R$ as in Proposition \ref{Cohen}.
We denote the finite morphism induced by the inclusion $A \subseteq R$ by $\pi : \Spec R \to \Spec A$.
We fix $K_A=0$ and $K_R= \Ram_\pi$.
\end{setup}

In order to prove Theorem \ref{main}, we first perturb $\D$ as follows.

\begin{lem}\label{perturb}
	Suppose that we are in the setting of Set up \ref{3.12}.
There exists a positive integer $q=p^e$, an element $f$ in $R$, and effective $\Q$-Weil divisors $\{ \D_n \}_{n \ge 1}$ on $R$ which satisfy the following properties.
       	\begin{enumerate}
       	\item $\D \le \dots \le \D_2 \le \D_1$.
       	\item $K_R+\D_n=\Div(f^{q^{n-1}})/(q^n-1)$ for all $n \ge 1$.
       	\item $\tau(R, \D)= \tau(R, \D_n)$ for all $n \ge 1$.
       	\end{enumerate}
        \end{lem}
\begin{proof}
       	Since $K_R=\Ram_\pi$ is effective (see \cite[Definition 4.5]{ST1}), for every positive number $\epsilon$, the divisor $\D_\epsilon \coloneqq \D+ \epsilon (K_R+\D)$ satisfies $\D_\epsilon \ge \D$.
        By Proposition \ref{fpt}, there exists $0 \le \epsilon \ll 1 $, $q'=p^{e'}$, and $f' \in R$ such that 
  	\[\tau(R, \D_\epsilon)= \tau(R, \D) \textup{ and } K_R+\D_\epsilon= \frac{\Div(f')}{q'}.\]
        For every positive integer $n$, we define 
        \[\D'_n \coloneqq \frac{\Div((f')^{(q')^{n-1}})}{(q')^n-1} -K_R.\]
        Since $\lim_{n \to \infty}\D'_n =\D_\epsilon$, there is $m>0$ such that $\tau(R, \D'_m)= \tau(R, \D_\epsilon)$.
        We define $f \coloneqq (f')^{(q')^{m-1}}$ and $q \coloneqq (q')^m$.
        Then 
        \[\D_n \coloneqq \frac{\Div(f^{q^{n-1}})}{q^n-1} - K_R\]
        satisfies the statement of this lemma.
\end{proof}
        
\begin{prop}\label{main'}
	Suppose that we are in the setting of Set up \ref{3.12} and take $q=p^e$ as in Lemma \ref{perturb}.
Assume that $n \ge 1$ and $a \in A$ satisfies $\ord_A(a)<q^{n-1}$.
	Then 
        \[\tau(R, \D+ \frac{\Div(a)}{q^n})=\tau(R,\D).\]
\end{prop}

\begin{proof}
	We have the unique decomposition 
        \[a=\sum_{\mu \in I_{en}} a_\mu \cdot x^\mu \textup{ with } a_\mu \in A^{q^n}\cdot k.\]
        We define 
        \[\mu_0 \coloneqq \min \left\{ \mu \in I_{en} \mid a_\mu \in A^\times \right\}.\]
        Since $\ord_A(a)<q^{n-1}$, $\mu_0$ exists and $\deg \mu_0 <q^{n-1}$.
        Let $u \coloneqq \overline{a_{\mu_0}} \in k=A/\m_A$.
        By replacing $a$ by $a\cdot u^{-1}$, we can assume that $a_{\mu_0} \equiv 1 \pmod{ \m_A}$.
        
        Fix $\alpha \in \tau(R, \D_n)$.
        Since $\tau(R, \D_n)=\tau(R, \D_1)=\phi_{e,\D}(F^e_* \tau( X, \D_1))=\Tr_{F^e_R, K_R}(F^e_*( f \cdot \tau(R, \D_n) ))$, 
        there exists $g \in f\cdot \tau(R, \D_n)$ such that $\alpha=\Tr_{F^e_R, K_R}(F^e_* g)$.
	Let us decompose 
        \[g= \sum_{\lambda \in I_e} g_\lambda \cdot x^\lambda \textup{ with } g_\lambda \in L^q \cdot k,\]
	and 
        \[g_\lambda= \sum_{i=1}^{l} g_{\lambda, i}^q \cdot t_i \textup{ with } g_{\lambda, i} \in L,\]
	where $\vt=(t_1, \dots, t_l)$ is a basis of $k$ over $k^q$.
        For every $\lambda \in I_e$ and $1 \le i \le l$, we have 
        \[g_{\lambda,i} = p_{\vt, (\mu_e, i)}(g x^{\mu_e-\lambda})
         \in  \Tr_{F^e_R, K_R}(F^e_*(f \tau(R, \D_n)))
         =  \tau(R, \D_n),
        \]
        and hence $g_\lambda \in \tau(R, \D_n)^{[q]} \cdot k$ for every $\lambda \in I_e$.
        Furthermore, we have $\alpha = g_{\mu_e, 1}$.
\begin{cl}\label{computation}
       	For all elements $\lambda$ in $I_e$ and $ 1 \le i \le l$, we have 
        \[g_{\lambda, i} \in \tau(R, \D_n+\frac{\Div(a)}{q^n})+ \m \cdot \tau(R, \D_n).\]
\end{cl}
If this claim is true, then we have the inclusion
        \[\tau(R, \D_n) \subseteq \tau(R, \D_n+ \frac{\Div(a)}{q^n})+ \m \cdot \tau(R, \D_n).\]
        Hence, we have $\tau(R, \D_n)= \tau(R, \D_n+ \Div(a)/q^n)$ by Nakayama.
Therefore, in order to finish the proof, it is enough to check the assertion in the claim.
  
Let $h \coloneqq g^{q^{n-1}} \cdot a \in R$ and decompose $h= \sum_{\xi} h_\xi x^\xi$ with $h_\xi \in L^{q^n} \cdot k$.
	Then we have 
	\[
	h=\sum_{\lambda \in I_e, \mu \in I_{en}}g_\lambda^{q^{n-1}} a_\mu x^{q^{n-1}\lambda+\mu}
	=\sum_{\lambda \in I_e, \mu \in I_{en}}(g_\lambda^{q^{n-1}} a_\mu x^{q^{n-1} \cdot \alpha (q^n \lambda+\mu)}) \cdot x^{\beta(q^{n-1}\lambda+\mu)},
	\]
	where for every $\eta \in \N^d$, we define $\alpha(\eta) \in \N^d$ and $\beta(\eta) \in I_{en}$ as the unique pair such that $\eta= q^n \alpha(\eta) + \beta(\eta)$.
	Consequently, we have 
        \[h_\xi= \sum_{\lambda, \mu} g_\lambda^{q^{n-1}} a_\mu x^{q^n \alpha(q^{n-1}\lambda+\mu)}, \]
	where $\lambda \in I_e$ and $\mu \in I_{en}$ run through all elements such that $\beta(q^{n-1}\lambda + \mu)= \xi$.
	
	Fix $\lambda_0 \in I_e$ and define $\xi_0 \coloneqq q^{n-1}\lambda_0 + \mu_0 \in \N^d$.
	Since $\deg \mu_0 <q^{n-1}$, we have $\xi_0 \in I_{en}$.
	We define an ideal $J \subseteq R$ as the ideal generated by all $g_{\lambda,i}$ with $\lambda < \lambda_0$ and $1 \le i \le l$.
	For every $\lambda \in I_e$ and $\mu \in I_{en}$, we have
	\[
	g_\lambda^{q^{n-1}} a_\mu x^{q^n \cdot \alpha}  \in \left\{ \begin{array}{ll}
	J^{[q^n]} \cdot k & \textup{(when $\lambda< \lambda_0$)}\\
	(\m_R \cdot \tau(R, \D_n))^{[q^n]} \cdot k & \textup{(when $\mu<\mu_0$ or $\alpha \neq 0$)} \\
	g_\lambda^{q^{n-1}} + (\m_R \cdot \tau(R, \D_n))^{[q^n]} \cdot k & \textup{(when $\mu=\mu_0$ and $\alpha=0$ )}
	\end{array} \right.,
	\]
        where $\alpha=\alpha(q^{n-1}\lambda+\mu) \in \N^d$.
	Hence we have $h_{\xi_0}- g_{\lambda_0}^{q^{n-1}} \in (J+\m_R \cdot \tau(R, \D_n))^{[q^n]} \cdot k$.
	
	Let $\mathbf{s}=(s_1, \dots ,s_{l'})$ be a basis of $k$ over $k^{q^n}$ such that $s_i=t_i^{q^{n-1}}$ for all $1 \le i \le l$.
	We can decompose $h_\xi= \sum_{1 \le i \le l'} h_{\xi, i}^{q^n} s_i$ with $h_{\xi,i} \in L$.
	By Lemma \ref{compose ideal 2}, we have 
	\[h_{\xi_0, i}- g_{\lambda_0,i} \in J+\m_R \cdot \tau(R, \D_n) \textup{ for every } 1 \le i \le l.\]
	On the other hand, we have 
	\begin{eqnarray*}
	h_{\xi_0, i}= p_{\mathbf{s}, (\mu_{en}, i)}(h \cdot x^{\mu_{en}-\xi_0}) & \in & \Tr_{F^{en}_R, K_R}(af^{q^{n-1}} \cdot \tau(R, \D_n)) \\
	& =& \phi_{en, \D_n}(a \cdot \tau(R, \D_n))\\
	& =& \tau(R, \D_n+ \frac{\Div(a)} {q^n}).
	\end{eqnarray*}
	Hence we have 
	\[g_{\lambda_0, i} \in \tau(R, \D_n+\frac{\Div(a)}{q^n}) + J+ \m_R \cdot \tau(R, \D).\]
	By induction on $\lambda_0$, we complete the proof of the claim, and hence of the proposition.
\end{proof}

We now give the proof of Theorem \ref{main}.
\begin{proof}[Proof of Theorem \ref{main}]
	Let $(R,\m)$, $\D$, $k$, $A$, be as in Set up \ref{3.12} and $q=p^e$ be as in Lemma \ref{perturb}.
        Let $M>0$ be as in Proposition \ref{norm} for the extension $A \subseteq R$.
        Then, for all $n \ge 1$ and $r \in R$ with $\ord_R(r) < q^{n-1}/M$,
        we have $\ord_A(\rN(r)) <q^{n-1}$.
        It follows from Proposition \ref{main'} that 
        \[\tau(R, \D+\frac{\Div(\rN(r))}{q^n})=\tau(R, \D).\]
        Since $\rN(r) \in r\cdot R$, we have  the inclusion
        \[\tau(R, \D+ \frac{\Div(\rN(r))}{q^n}) \subseteq \tau(R, \D+ \frac{\Div(r)}{q^n}).\]
        This completes the proof.
\end{proof}

\begin{cor}\label{main''}
	Let $X$ be an $F$-finite Noetherian normal scheme, $\D$ be an $\R$-Weil divisor on $X$, and $P$ be a point of $X$.
        Then there exists $\delta>0$ which satisfies the following properties.
        \begin{enumerate}
        \item Let $D$ be an effective Cartier divisor on $X$ and $t$ be an element in $\R_{ > 0}$.
	If $t \cdot \ord_P(D) < \delta$, then $\tau(X, \D)_P=\tau(X, \D+tD)_P$.
        \item Let $\fa \subseteq \sO_X$ be a coherent ideal and $t$ be an element in $\R_{ > 0}$.
        If $t \cdot \ord_P(\fa) < \delta$, then $\tau(X, \D)_P=\tau(X, \D, \fa^t)_P$.
        \item Let $\fa_\bullet$ be a graded sequence of ideals and $t$ be an element in $\R_{ > 0}$.
        If $t \cdot \ord_P(\fa_\bullet) < \delta$, then $\tau(X, \D)_P=\tau(X, \D, \fa_\bullet^t)_P$.
        \end{enumerate}
\end{cor}

\begin{proof}
	We may assume that $X=\Spec R$ for $R=\sO_{X,P}$ and $\D$ is effective.
        By Proposition \ref{non gor test}, we may assume that $K_X+\D$ is a $\Q$-Cartier $\Q$-divisor.
        
	Let $q$ and $M$ be as in Theorem \ref{main}.
        Then $\delta=1/(Mq^2)$ satisfies the statement (1).
        In (2), there exists $f \in \fa$ such that $t \cdot \ord_R(f) < \delta$ and hence $\tau(X, \D)= \tau(X, \D+t \Div(f))$.
        By Proposition \ref{test basic} (4), we have $\tau(X, \D+ t \Div(f)) \subseteq \tau(X, \D, \fa^t)$. This shows (2).
        The statement of (3) follows from (2).
\end{proof}

For a Noetherian local ring $(R, \m)$, we denote the \emph{Hilbert-Samuel multiplicity} of $R$ by $\e(R)$ (see \cite[p.108]{Mat}).
For a Noetherian scheme $X$ and a (not necessarily closed) point $x \in X$, we define the \emph{multiplicity} of $X$ at $x$ by $\mult_x(X) := \e(\sO_{X,x})$.
Moreover, for a closed subscheme $Y \subseteq X$, we define
\[
\mult_x(Y) := \left\{ \begin{array}{ll}
\e(\sO_{Y, x}) & \textup{( when } x \in Y) \\
0 & \textup{( when } x \not\in Y )
\end{array}\right. .
\]

\begin{defn}\label{multiplicity}
Let $X$ be a Noetherian integral scheme and $x$ be a point of $X$.
Let $ D= \sum_{i=1}^n a_i D_i$ be an $\R$-Weil divisor on $X$, where $D_i \subseteq X$ is a prime divisor for every $i$.
Then we define $\mult_x(D) := \sum_{i=1}^n a_i \mult_x(D_i)$.
\end{defn}

By Corollary \ref{main''} (1), we obtain the main theorem.

\begin{cor}[Main Theorem]\label{main'''}
	Let $X$ be an $F$-finite Noetherian normal scheme, $\D$ be an $\R$-Weil divisor on $X$, and $P$ be a point of $X$.
	Then there exists $\delta>0$ such that for every effective $\R$-Cartier $\R$-divisor $E$ on $X$ with $\mult_P(E)< \delta$, we have 
        \[\tau(X, \D+ E)_P=\tau(X, \D)_P.\]
\end{cor}
\begin{proof}
After adding to $E$ a suitable effective $\R$-Cartier $\R$-divisor, we may assume that $E$ is a $\Q$-Cartier $\Q$-divisor.
	Take a positive integer $r$ such that $D \coloneqq rE$ is an integral Cartier divisor.
        By \cite[Theorem 14.7 and Theorem 14.9]{Mat}, we have $\ord_P(D) \le \mult_P(D)$.
        Hence, we have $\mult_P(E)=\mult_P(D)/r \ge \ord_P(D)/r$ and we can apply Corollary \ref{main''} (1).
\end{proof}
\subsection{Applications to graded sequences of ideals}

\begin{prop}\label{asymp jump}
	Let $(X , \D)$ be an $F$-finite log $\Q$-Gorenstein pair and $P$ be a point of $X$.
        For a graded sequence $\fa_\bullet$, the following are equivalent:
        \begin{enumerate}
        \item $\bigcap_{m \ge 1} \J(X, \D, \fa_\bullet^m)_P=0$.
        \item $\bigcap_{m \ge 1} \tau(X, \D, \fa_\bullet^m)_P=0$.
        \item $\exists m \ge 1$, $\tau(X, \D, \fa_\bullet^m)_P \neq \tau(R, \D)_P$.
        \item $\ord_P(\fa_\bullet)>0$.
        \item For all $v$ such that $\cen_X(v)=\overline{\{P\}}$, we have $v(\fa_\bullet)>0$.
        \item There exists $v$ such that $\cen_X(v)=\overline{\{P\}}$ and $v(\fa_\bullet)>0$.
        \end{enumerate}
\end{prop}
\begin{proof}
	We have the implication (1) $\Rightarrow$ (2) since we have the inclusion $\tau(X, \D, \fa^m_\bullet) \subseteq \J(X, \D, \fa^m_\bullet)$ for all $m$ (see \cite[Theorem 2.13]{Tak}, cf. \cite[Proposition 4.1]{BST}).
	The implication (2) $\Rightarrow$ (3) is trivial.
	The implication (3) $\Rightarrow$ (4) follows from Corollary \ref{main''} (3).
	Since $\ord_P(r) \le v(r)$ for all $r \in \sO_{X,P}$, we have the implication (4) $\Rightarrow$ (5).
	The implication (5) $\Rightarrow$ (6) is trivial.
        
	For the implication (6) $\Rightarrow$ (1), take a proper birational morphism $\phi:Y \to X$ and a prime divisor $E \subseteq Y$ such that $v(\fa_\bullet)>0$, where $v \coloneqq \ord_E$.
	For every $0 \neq f \in L=\Frac(R)$, there exists $m >0$ such that $v(f) < \ord_E(\D_Y)-1 + m \cdot v(\fa_\bullet)$.
        Then, for every $n >0$, we have $v(f) < \ord_E(\D_Y)-1 + m \cdot v(\fa_n)/n$ and hence $f \not \in \J(X, \D, \fa_n^{m/n})_P$.
        This means $f \not\in \J(X, \D, \fa_\bullet^m)_P$.
\end{proof}

Let $X$ be an $F$-finite Noetherian normal scheme and $\fa_\bullet$ be a non-zero graded sequence of ideals.
We define 
\[
\NNA(\fa_\bullet) \coloneqq \bigcup_{v} \cen_X(v) \subseteq X,
\]
where $v$ runs through all divisorial valuations such that $v(\fa_\bullet)>0$.

For a $\Q$-Weil divisor $\D$ on $X$, we define
\begin{eqnarray*}
\Sigma'_\D(\fa_\bullet) & \coloneqq & \left\{ x \in X \mid \bigcap_m \tau(X, \D, \fa_\bullet^m)_x =0 \right\} \subseteq X \hspace{0.5cm} \textup{and} \\
\Sigma_\D(\fa_\bullet) & \coloneqq & \left\{ x \in X \mid \exists m>0, \tau(X, \D, \fa_\bullet^m)_x \subsetneq \tau(X, \D)_x \right\} \subseteq X.
\end{eqnarray*}
When $\D=0$, we write simply $\Sigma'(\fa_\bullet)$ and $\Sigma(\fa_\bullet)$.

\begin{cor}\label{jump locus}
	With the above notation, we have
	\[
	\Sigma'_\D(\fa_\bullet) = \Sigma_\D(\fa_\bullet) =\NNA(\fa_\bullet).
	\]
        In particular, we have $\NNA(\fa_\bullet)=\Sigma(\fa_\bullet)$.
\end{cor}

\begin{proof}
	Since the definition of every locus is local, we may assume that $X$ is affine.
        
        First, we assume that $K_X+\D$ is a $\Q$-Cartier $\Q$-divisor.
        The inclusion $\Sigma'_\D(\fa_\bullet) \subseteq \Sigma_\D(\fa_\bullet)$ is trivial and the inclusion $\Sigma_\D(\fa_\bullet) \subseteq \NNA(\fa_\bullet)$ follows from Proposition \ref{asymp jump}.
	If $x \in \NNA(\fa_\bullet)$, there exists a divisorial valuation $v$ such that $v(\fa_\bullet)>0$ and $x \in \cen_X(v)$.
        Let  $y \in X$ be the generic point of $\cen_X(v)$.
        By Proposition \ref{asymp jump}, we have $y \in \Sigma'_\D(\fa_\bullet)$.
        Since $\Sigma'_\D(\fa_\bullet)$ is stable under specialization, we have $x \in \Sigma'_\D(\fa_\bullet)$.
        
        Suppose now that $\D$ is an arbitrary $\Q$-Weil divisor.
        Then the inclusion $\Sigma'_\D(\fa_\bullet) \subseteq \Sigma_\D(\fa_\bullet)$ is trivial.
        For the inclusion $ \Sigma_\D(\fa_\bullet) \subseteq \Gamma(\fa_\bullet)$, we take $x \not\in \Gamma(\fa_\bullet)$.
Then, we have $x \not\in \Sigma_{\D'}(\fa_\bullet)$ for every $\Q$-Weil divisor on $X$ such that $K_X+ \D'$ is a $\Q$-Cartier $\Q$-divisor.
By Proposition \ref{non gor test}, we have 
\[ \tau(X, \D, \fa^m_\bullet)_x= \sum_{\D'} \tau(X, \D', \fa^m_\bullet)_x = \sum_{\D'} \tau(X, \D')_x =\tau(X, \D)_x, \]
where $\D'$ runs through all $\Q$-Weil divisors on $X$ such that $\D' \ge \D$ and $(p^e-1)(K_X+\D')$ is an integral Cartier divisor for some $e>0$.
Hence we have the inclusion $\Sigma_\D(\fa_\bullet) \subseteq \NNA(\fa_\bullet)$.
        
        In order to show the inclusion $\NNA(\fa_\bullet) \subseteq \Sigma'_\D(\fa_\bullet)$, we take $x \not \in \Sigma'_\D(\fa_\bullet)$.
        Let $\D'$ be a $\Q$-Weil divisor on $X$ such that $K_X+ \D'$ is a $\Q$-Cartier $\Q$-divisor and $\D' \le \D$.
        Since we have the inclusion $\tau(X, \D, \fa_\bullet)_x \subseteq \tau(X, \D', \fa_\bullet)_x$, we have $x \not \in \Sigma'_{\D'} (\fa_\bullet)=\NNA(\fa_\bullet)$.
\end{proof}

\begin{cor}{\textup{(\cite[Theorem 4.2]{CL})}}
With the above notation, let $v$ be a divisorial valuation.
Then $v(\fa_\bullet)>0$ if and only if $\cen_X(v) \subseteq \NNA(\fa_\bullet)$.
\end{cor}

\begin{proof}
	We assume that $\cen_X(v) = \overline{\{ P \}} \subseteq \NNA(\fa_\bullet)$.
        We fix a $\Q$-Weil divisor $\D$ on $X$ such that $K_X+\D$ is a $\Q$-Cartier $\Q$-divisor.
        Since $P \in \NNA(\fa_\bullet)$, we have $P \in \Sigma_\D(\fa_\bullet)$.
        By Proposition \ref{asymp jump}, we have $v(\fa_\bullet)>0$.
\end{proof}

\section{Application to non-nef loci}

In this section, we first show a uniform global generation result involving test ideals (Proposition \ref{grob}).
Combining it with the main theorem (Corollary \ref{main'''}), we prove Conjecture \ref{conj} for strongly $F$-regular varieties (Corollary \ref{Nakayama}). 

From now on, we always assume that $X$ is a normal projective variety over an $F$-finite field $k$ of characteristic $p>0$.

\subsection{Global generation}
In this subsection, we recall a result on global generation and its applications proved in \cite[Section 4 and 5]{Mus}.
In \cite{Mus}, it is assumed that $X$ is smooth, $k$ is algebraically closed, and $\D=0$.
However, we can relax the assumptions.
\begin{prop}[\textup{\cite[Theorem 4.1]{Mus}}]\label{grob}
	Suppose that $\D$ is an effective $\Q$-Weil divisor on $X$ such that $K_X+\D$ is a $\Q$-Cartier $\Q$-divisor.
        Let $D$ and $L$ be Cartier divisors on $X$, $H$ be an ample and free divisor on $X$, $\lambda$ be an element in $\R_{ \ge 0}$, and $d$ be a positive integer such that $d>\dim X$.
	If $L-(K_X+\D+\lambda \cdot D)$ is ample, then the following properties hold.
	\begin{enumerate}
	\item $\tau(X, \D, \lambda \cdot |D|) \otimes \sO_X(L+dH)$ is globally generated.
	\item $\tau(X, \D, \lambda \cdot ||D||) \otimes \sO_X(L+dH)$ is globally generated.
	\end{enumerate}
\end{prop}

\begin{proof}
	(2) follows from (1).
        We prove (1). The proof is essentially the same as the proof in \cite{Mus}.
        
        By enlarging $\lambda$ and $\D$, we may assume that $\lambda \in \Q$ and $(q-1)(K_X+\D)$ is an integral Cartier divisor for some $q=p^e>1$.
        By Proposition \ref{phi acc}, we have 
        \[\widetilde{\phi}_{en, \D}(F^{en}_*(\mathfrak{b}(|D|)^{\lceil \lambda q^n \rceil} \tau(X, \D) \cdot \sL_{en, \D}))= \tau(X, \D, \lambda |D|),\]
         for all $n \gg 0$.
        Since $H^0(X, D) \otimes \sO_X(-D) \to \mathfrak{b}(|D|)$ is surjective, we have a surjective morphism 
        \[W_n \otimes \sO_X(- \lceil \lambda q^n \rceil D)) \to \mathfrak{b}(|D|)^{\lceil \lambda q^n \rceil},\] 
        where $W_n = \mathrm{Sym}^{\lceil \lambda q^n \rceil} (H^0(X, D))$.
        Hence, we have a surjective morphism 
        \begin{eqnarray*}
        W_n \otimes F^{en}_*(\sO_X(- \lceil \lambda q^n \rceil D + (1-q^n)(K_X+\D)+q^n(L+dH)) \tau(X, \D)) \\
        \to \tau(X, \D, \lambda |D|) \otimes \sO_X(L+dH).
        \end{eqnarray*}
        We write $E_n :=- \lceil \lambda q^n \rceil D + (1-q^n)(K_X+\D) +q^n L$.
        By \cite[p.11, Claim]{Mus}, there exists a divisor $T$ on $X$ such that $E_n -T$ is nef for all $n \ge 1$.
        By Fujita's vanishing theorem, we have 

        \[H^i(X, F^{en}_*(\sO_X(E_n + q^n d H)) \tau(X, \D)) \otimes \sO_X(-iH))=0,\]
         for all $i>0$ and for all $n \gg 0$.
        This shows 
        \[W_n \otimes F^{en}_*(\sO_X(E_n+q^n d H)) \tau(X, \D))\]
         is globally generated for all $n \gg 0$ and this completes the proof.
\end{proof}

By using Proposition \ref{grob}, we can show the following three properties.
The reader is referred to \cite[Proposition 5.1, 5.2, 5.3]{Mus} for proofs.

\begin{cor}[\textup{\cite[Proposition 5.1]{Mus}}]
	Let $\D$ be a $\Q$-Weil divisor on $X$, $D$ and $E$ be big $\Q$-Cartier $\Q$-divisors on $X$ and $t$ be an element in $\R_{ \ge 0}$.
        If $D \equiv _ {\num} E$, then $\tau(X, \D, t \cdot ||D||)= \tau(X, \D, t \cdot ||E||)$.
\end{cor}

\begin{cor}[\textup{\cite[Proposition 5.2]{Mus}}]
	Let $\D$ be a $\Q$-Weil divisor on $X$ such that $K_X+\D$ is a $\Q$-Cartier $\Q$-divisor and $D$ be a $\Q$-Cartier $\Q$-divisor on $X$.
        Then the set of $F$-jumping numbers 
        \[\left\{ t \in \R_{\ge 0} \mid \forall s>t , \tau(X, \D, s \cdot ||D||) \subsetneq \tau(X, \D, t \cdot ||D||)  \right\} \]
        is a discrete set.
\end{cor}

\begin{propdef}[\textup{\cite[Proposition 5.3]{Mus}}]
	Let $\D$ be a $\Q$-Weil divisor on $X$ such that $K_X+\D$ is a $\Q$-Cartier $\Q$-divisor, $D$ be an $\R$-Cartier $\R$-divisor on $X$ and $\lambda$ be an element in $\R_{\ge 0}$.
        Then there exists a unique minimal element among all ideals of the form $\tau(X, \D, \lambda \cdot ||D+A||)$,
        where $A$ runs through all ample $\R$-divisors on $X$ such that $D+A$ is a $\Q$-Cartier $\Q$-divisor.
        We denote this ideal as $\tau_+(X, \D, \lambda \cdot ||D||)$.
        Furthermore, there exists an open neighborhood $\mathscr{U}$ of the origin in $\mathrm{N}^1(X)_\R$ such that
        \[
        \tau_+(X, \D, \lambda \cdot ||D||)= \tau (X, \D, \lambda \cdot ||D+A||)
        \]
        for every ample $\R$-divisor $A$ on $X$ such that $D+A$ is a $\Q$-Cartier $\Q$-divisor and the class of $A$ is in $\mathscr{U}$.
\end{propdef}

\subsection{Nakayama's theorem about non-nef loci}

\begin{thm}
Let $X$ be a normal projective variety over an $F$-finite field, 
let $\D$ be an effective $\Q$-Weil divisor on $X$ such that $K_X+\D$ is a $\Q$-Cartier $\Q$-divisor,
let $D$ be a big $\Q$-Cartier $\Q$-divisor on $X$, 
and let $P$ be a point of $X$.
Assume that $(X, \D)$ is strongly $F$-regular at $P$.
Then the following are equivalent:
\begin{enumerate}
\item $P \in \B_-(D)$.
\item $P \in \bigcup_{m \ge 1} Z( \tau(X, \D, m||D||))$.
\item There exists a divisorial valuation $v$ such that $\overline{\{ P \}} \subseteq \cen_X(v)$ and $v(||D||)>0$.
\end{enumerate}
\end{thm}

\begin{proof}
	The proof is essentially the same as \cite[Proposition 2.8]{ELMNP} (cf. \cite[Theorem 6.2]{Mus} and \cite[Section 4]{CdB}).
        The implication (3) $\Rightarrow$ (1) is obvious.
Let $\fa_\bullet :=\fa_\bullet(D)$ be as in Remark \ref{||D||}.
If the assertion in (2) holds, then we have $P \in \Sigma_\D (\fa_\bullet)$ since $\tau(X, \D)_P=\sO_{X,P}$. 
By Corollary \ref{jump locus}, we have $P \in \Gamma(\fa_\bullet)$, which shows the implication (2) $\Rightarrow$ (3).
        For the implication (1) $\Rightarrow$ (2), we assume $P \notin \bigcup_{m \ge 1} Z( \tau(X, \D, m||D||))$.
        By Proposition \ref{grob}, there exists an ample divisor $A$ on $X$ such that $\tau(X, \D, m \cdot ||D||) \otimes \sO_X(mD+A)$ is globally generated for all $m>0$.
        Therefore, $\sO_X(mD+A)$ is globally generated at $P$ for all $m>0$, which implies $P \notin \B_{-}(D)$.
        
\end{proof}

\begin{cor}\label{Nakayama Q}
Let $X$ be a normal projective variety over an $F$-finite field,
$D$ be a big $\Q$-Cartier $\Q$-divisor on $X$ and $P$ be a point of $X$.
Assume that $X$ is strongly $F$-regular at $P$.
Then the following are equivalent:
\begin{enumerate}
\item $P \in \B_-(D)$.
\item $P \in \bigcup_{m \ge 1} Z( \tau(X, m||D||))$.
\item There exists a divisorial valuation $v$ such that $\overline{\{ P \}} \subseteq \cen_X(v)$ and $v(||D||)>0$.
\end{enumerate}
\end{cor}

\begin{proof}
	By Proposition \ref{non gor test}, there exists 
        an effective $\Q$-Weil divisor $\D$ on $X$ such that $K_X+\D$ is a $\Q$-Cartier $\Q$-divisor
        and $(X, \D)$ is strongly $F$-regular at $P$.
        Hence (1) and (3) are equivalent.
        By Corollary \ref{jump locus}, (2) and (3) are equivalent.
\end{proof}

\begin{cor}[\textup{\cite[Corollary 4.7]{CdB}}]\label{Nak general}
Let $X$ be a normal projective variety over an $F$-finite field and $D$ be an $\R$-Cartier $\R$-divisor on $X$.
Then we have 
\[
\NNef(D) \setminus Z( \tau(X)) = \B_{-}(D) \setminus Z( \tau(X))
\]
\end{cor}

\begin{proof}
	When $D$ is not pseudo-effective, $\NNef(D)=\B_{-}(D)=X$.
        Assume that $D$ is pseudo-effective.
	Let $\{A_n\}_{n \in \N}$ be a sequence of ample $\R$-Cartier $\R$-divisors on $X$ such that $D+ A_n$ is a $\Q$-Cartier $\Q$-divisor for every $n$ and $\lim_{n \to \infty}A_n= 0$ in $\rN^1(X)_\R$.
        
        Then we have $\B_{-}(D)= \bigcup_n \B_-(D+A_n)$ (\cite[Proposition 2.2]{Mus}) 
        and $\NNef(D)= \bigcup_n \NNef(D+A_n)$.        
        Hence we can reduce to the case when $D$ is a big $\Q$-Cartier $\Q$-divisor and this is Corollary \ref{Nakayama Q}.
\end{proof}

\begin{cor}[Theorem \ref{Nak}]\label{Nakayama}
Let $X$ be a normal projective variety over an $F$-finite field.
Assume that $Z(\tau(X))$ has dimension at most 0.
Then, for every $\R$-Cartier $\R$-divisor $D$ on $X$, we have
\[
\NNef(D)=\B_{-}(D).
\]
\end{cor}

\begin{proof}
The proof is similar to the proof of \cite[Corollary 4.9]{CdB}.
	The inclusion $\NNef(D) \subseteq \B_{-}(D)$ is obvious.
Let $\{A_n\}_{n \in \N}$ be a sequence of $\R$-Cartier $\R$-divisors as in the proof of Corollary \ref{Nak general}.
Then we have $\B_{-}(D)= \bigcup_n \B(D+A_n)$ by \cite[Proposition 2.2]{Mus}. 
Since every $\B(D+A_n)$ has no isolated points (see \cite[Proposition 1.1]{ELMNP2}), we have the inclusion $\B_{-}(D) \subseteq \NNef(D)$ by Corollary \ref{Nak general}.
\end{proof}


\end{document}